\documentclass[sn-mathphys-num]{sn-jnl}% Math and Physical Sciences Numbered Reference Style 
\usepackage{graphicx}%
\usepackage{multirow}%
\usepackage{amsmath,amssymb,amsfonts}%
\usepackage{amsthm}%
\usepackage{mathrsfs}%
\usepackage[title]{appendix}%
\usepackage{xcolor}%
\usepackage{textcomp}%
\usepackage{manyfoot}%
\usepackage{booktabs}%
\usepackage{algorithm}%
\usepackage{algorithmicx}%
\usepackage{algpseudocode}%
\usepackage{listings}%
\newtheorem{theorem}{Theorem}[section]
\newtheorem{corollary}[theorem]{Corollary}
\newtheorem{lemma}[theorem]{Lemma}
\newtheorem{proposition}[theorem]{Proposition}
\theoremstyle{definition}
\newtheorem{definition}[theorem]{Definition}
\newtheorem{example}[theorem]{Example}
\theoremstyle{remark}
\newtheorem{Remark}[theorem]{Remark}
\numberwithin{equation}{section}

\begin{document}

\title[Constructing Parseval Fusion Frames via Operators]{Constructing Parseval Fusion Frames via Operators}


\author[1]{\fnm{Ehsan} \sur{Ameli}}\email{eh.ameli@hsu.ac.ir; ameliehsan73@gmail.com}

\author*[1]{\fnm{Ali Akbar} \sur{Arefijamaal}}\email{arefijamaal@hsu.ac.ir}

\author[2,1]{\fnm{Fahimeh} \sur{Arabyani Neyshaburi}}\email{fahimeh.arabyani@gmail.com}

\affil[1]{\orgdiv{Department of Mathematics and Computer Sciences}, \orgname{Hakim Sabzevari University}, \state{Sabzevar}, \country{Iran}}

\affil[2]{\orgdiv{Department of Mathematical Sciences}, \orgname{Ferdowsi University of Mashhad}, \city{Mashhad}, \country{Iran}}


\abstract{This article explores the problem of modifying the subspaces of a fusion frame in order to construct a Parseval fusion frame. In this respect, the notion of scalability is extended to the fusion frame setting. Then, scalable fusion Riesz bases are characterized, and a concrete form for scalable 1-excess fusion frames is obtained. Furthermore, it is shown that 1-excess dual fusion frames of a fusion Riesz basis are not scalable. Finally, several examples are exhibited to confirm the acquired results.}




\maketitle

\section{Introduction and Preliminaries}

A Parseval fusion frame is a collection of orthogonal projection operators such that their sum equals the identity operator \cite{frame of subspace}. These types of fusion frames exhibit robustness against additive noise and erasures \cite{scaling weights, Weaving Hilbert space, Norm retrieval algorithms,Cahill-Fickus,CasazzaFickus, Rahimi}, making them well-suited for emerging real-world applications in communications and distributed sensing \cite{excess of fusion, Arabyani dual, A.A.SH}. Therefore, the development of Parseval fusion frames provides a substantial framework for numerous applications, including sensor networks, coding theory, filter bank theory, signal and image processing, wireless communications, and many other fields \cite{Cahill-Fickus, CasazzaFickus, Casazza17, convex geometry}. A fundamental question is how to construct a Parseval fusion frame from a given fusion frame. In addressing this challenge, the authors in \cite{scaling weights} proposed a novel approach that involves scaling weights. In this research, motivated by the results obtained in \cite{scaling weights,Casazza17,dual scalable frames,Gitta 13,convex geometry,Rahimi scalable}, we aim to construct Parseval fusion frames by modifying subspaces as well as adjusting weights. Building upon this perspective, we are seeking operators that allow the generation of Parseval fusion frames by implementing specific changes in the subspaces.

To facilitate comprehension for the reader, we present a preliminary review of the definitions and basic properties of fusion frames.
Throughout this paper, let $\mathcal {H}$ be a separable Hilbert space and $\mathcal {H}_n$ an $n$-dimensional Hilbert space. Moreover, let $I$ and $J$ be countable index sets and $I_{\mathcal {H}}$ be the identity operator on $\mathcal {H}$. We denote the set of all bounded operators on $\mathcal {H}$ by $B(\mathcal {H})$ and the range and null space of $T \in B(\mathcal {H})$ by $R(T)$ and $N(T)$, respectively. Furthermore, the transpose of a matrix $M$ is denoted by $M^t$. Suppose that $\{W_{i}\}_{i \in I}$ is a collection of closed subspaces of $\mathcal {H}$ and $ \{\omega_{i}\}_{i \in I} $ is a family of weights, i.e. $\omega_{i}>0 $ for all $i \in I$. Then $ \mathcal{W}=\{(W_{i},\omega_{i})\}_{i \in I} $ is called a \textit{fusion frame} \cite{frame of subspace} for $\mathcal {H}$ if there exist constants $0<A\leq B<\infty$ such that for all $f \in \mathcal{H},$
\begin{equation}\label{fusion frame def}
A \Vert f\Vert^{2}\leq \sum _{i \in I} \omega_{i}^2\Vert \pi_{W_{i}}f \Vert ^{2} \leq B\Vert f \Vert ^{2},
\end{equation} 
where $\pi_{W_{i}}$ is the orthogonal projection onto the subspace $W_{i}.$ The constants $A$ and $B$ are called the \textit{fusion frame bounds}. If we only have the upper bound in \eqref{fusion frame def}, then $ \mathcal{W} $ is said to be a \textit{fusion Bessel sequence}. A fusion frame is called \textit{A-tight} if $ A=B $, and \textit{Parseval} if $ A=B=1 $. If $ \omega_{i}=\omega $ for all $ i\in I $, then $ \mathcal{W}$ is called \textit{$ \omega $-uniform}. We abbreviate 1-uniform fusion frames as $\{W_{i}\}_{i \in I}$. A family of closed subspaces $\{W_{i}\}_{i \in I} $ is said to be a \textit{fusion orthonormal basis} when $\mathcal {H}$ is the orthogonal sum of the subspaces $ W_{i}$ and it is said a \textit{Riesz decomposition} of $\mathcal {H}$, if for every $ f \in \mathcal{H} $ there is a unique choice of $ f_{i}\in W_{i} $ such that $ f=\sum _{i \in I} f_{i} $. It is evident that every fusion orthonormal basis is a Riesz decomposition for $\mathcal {H}$, and also every Riesz decomposition is a 1-uniform fusion frame for $\mathcal {H}$ \cite{frame of subspace}. In addition, a family $\{W_{i}\}_{i \in I} $ of closed subspaces of $\mathcal {H}$ is a fusion orthonormal basis if and only if it is a 1-uniform Parseval fusion frame \cite{frame of subspace}. A family of closed subspaces $\{W_{i}\}_{i \in I} $ is called a \textit{fusion Riesz basis} whenever it is complete for $\mathcal {H}$ and there exist positive constants $C$ and $D$ such that for every finite subset $ J\subset I $ and arbitrary vector $ f_{j}\in W_{j}~ (j \in J)$, we have
\begin{equation*}\label{fusion Riesz basis}
C \sum _{j \in J}\Vert f_{j} \Vert ^{2} \leq \bigg\Vert \sum _{j \in J} \omega_{j} f_{j}\bigg\Vert ^2 \leq D\sum _{j \in J}\Vert f_{j} \Vert ^{2}.
\end{equation*}
For a Bessel sequence $\mathcal{W}=\{(W_{i},\omega_{i})\}_{i \in I}$, the \textit{synthesis operator} $ T_{\mathcal{W}}: \sum_{i \in I}\bigoplus W_{i} \rightarrow \mathcal {H}$ is defined by 
\begin{equation*}
T_{\mathcal{W}}(\{f_{i}\}_{i \in I})=\sum_{i \in I}\omega_{i}f_{i}, \quad \{f_{i}\}_{i \in I} \in \sum_{i \in I}\bigoplus W_{i},
\end{equation*}
where 
\begin{equation*}
\sum_{i \in I}\bigoplus W_{i}=\left\lbrace \{f_{i}\}_{i \in I}: f_{i}\in W_{i}~,~ \sum_{i \in I}\Vert f_{i} \Vert ^{2}< \infty \right\rbrace.
\end{equation*}
Furthermore, the \textit{fusion frame operator} $S_{\mathcal{W}}:\mathcal {H} \rightarrow \mathcal {H}$ defined by
\begin{equation*}
S_{\mathcal{W}}f=T_{\mathcal{W}}T^{*}_{\mathcal{W}}f=\sum_{i \in I}\omega_{i}^{2}\pi_{W_{i}}f,
\end{equation*}
is positive and self-adjoint. It is shown that for a fusion frame $\mathcal{W}$, the fusion frame operator $S_{\mathcal{W}}$ is invertible, and thus we have the following reconstruction formula \cite{frame of subspace}: 
\begin{equation*}
f=\sum_{i \in I}\omega_{i}^{2}S_{\mathcal{W}}^{-1}\pi_{W_{i}}f, \quad(f \in \mathcal {H}).
\end{equation*}
Obviously, $\mathcal{W}=\{(W_{i},\omega_{i})\}_{i \in I} $ is a Parseval fusion frame if and only if $S_{\mathcal{W}}=I_{\mathcal {H}}$. The family $\widetilde{\mathcal{W}}:=\left\lbrace (S_{\mathcal{W}}^{-1}W_{i},\omega_{i})\right\rbrace _{i \in I}$, which is also a fusion frame, is called the \textit{canonical dual} of $\mathcal{W}$. Generally, a fusion Bessel sequence $\mathcal{V}=\{(V_{i},\upsilon_{i})\}_{i \in I}$ is a dual of $\mathcal{W}$ if and only if \cite{Osgooei}
$T_{\mathcal{V}}\varphi_{\mathcal{VW}} T_{\mathcal{W}}^{*}=I_{\mathcal {H}},$ 
where the bounded operator $\varphi_{\mathcal{VW}}:\sum_{i \in I}\bigoplus W_{i}\rightarrow \sum_{i \in I}\bigoplus V_{i}$ is given by
\begin{equation*}
\varphi_{\mathcal{VW}}\big(\{f_{i}\}_{i \in I}\big)=\left\lbrace \pi_{V_{i}}S_{\mathcal{W}}^{-1}f_{i}\right\rbrace _{i \in I} .
\end{equation*}
Given a fusion frame $\mathcal{V}=\{V_{i}\}_{i \in I} $ with the synthesis operator $T_{\mathcal{V}}$, the \textit{excess} of $\mathcal{V}$ is defined as \cite{Excess 1}
\begin{equation*}
e(\mathcal{V})=\textnormal{dim}N(T_{\mathcal{V}}).
\end{equation*}

For a comprehensive survey on fusion frame theory and its applications, the reader is referred to \cite{Weaving Hilbert space, Norm retrieval algorithms,frame of subspace, 16, Excess 1,Nga,mitra sh.}. The following result establishes the relationship between local and global properties.

\begin{theorem}\label{Local theorem}
\cite{frame of subspace}  Let $\{W_{i}\}_{i \in I}$ be a family of closed subspaces of $\mathcal {H},$ and for every $i \in I,$ $ \omega _{i} >0 $ and $\{f_{i,j}\}_{j \in J_{i}}$ be a frame (Riesz basis) for $W_{i}$ with frame bounds $A_{i}$ and $B_{i}$ such that
\begin{equation*}
0 < A=\textnormal{inf}_{i \in I} A_{i}\leq \textnormal{sup}_{i \in I} B_{i}=B < \infty .
\end{equation*}
Then the following conditions are equivalent:
\begin{itemize}
\item[$(i)$] $\{(W_{i},\omega_{i})\}_{i \in I}$ is a fusion frame (fusion Riesz basis) for $\mathcal {H}$,
\item[$(ii)$] $\{\omega_{i}f_{i,j}\}_{i \in I, j \in J} $ is a frame (Riesz basis) for $\mathcal {H}$.
\end{itemize}
\end{theorem}

Finally, we list two known results that will be employed in the subsequent sections.
\begin{proposition}\cite{16} \label{Gavruta1} 
Let $W$ be a closed subspace of $\mathcal{H}$ and $U \in B(\mathcal{H})$ an invertible operator. Then the following are equivalent:
\begin{itemize}
\item[$(i)$] $\pi_{UW}U=U\pi_{W}.$
\item[$(ii)$]$U^{*}UW\subseteq W .$
\end{itemize}
Moreover, $\pi_{UW}=\pi_{UW}U^{*-1}\pi_{W}U^{*}$.
\end{proposition}

\begin{proposition}\cite{frame of subspace, mitra sh.} \label{mitra sh.} 
Let $\mathcal{W}=\{(W_{i},\omega_{i})\}_{i \in I} $ be a fusion frame for $\mathcal {H}$. Then the following are equivalent:
\begin{itemize}
\item[$(i)$] $\mathcal{W}$ is a fusion Riesz basis.
\item[$(ii)$] $S_{\mathcal{W}}^{-1}W_{i}\perp W_{j}$ for all $i , j \in I,~ i \neq j $.
\item[$(iii)$] $ \omega_{i}^{2}\pi_{W_{i}}S_{\mathcal{W}}^{-1}\pi_{W_{j}}=\delta_{i,j}\pi_{W_{j}}$ for all $i , j \in I $. 
\end{itemize}
\end{proposition}

The paper is organized as follows. In Section 2, we establish the concept of operator-scalability and investigate a method for constructing Parseval fusion frames. Subsequently, we present an operator condition that corresponds to the scalability of fusion frames. In particular, we provide a complete characterization of operator-scalable fusion Riesz bases. In Section 3, we study the operator-scalability of 1-excess fusion frames and demonstrate that 1-excess dual fusion frames of a fusion Riesz basis are not operator-scalable. Finally, we give some examples in order to confirm the results. 

\section{Operator-Scalability and Fusion Frame Properties}
In this section, we propose a method for constructing Parseval fusion frames that is based on altering the subspaces and scaling the weights. First of all, we provide some generalizations of Propositions \ref{Gavruta1} and \ref{mitra sh.} that will be used in the subsequent results.

\begin{lemma}\label{extend Gav}
Let $ W\subset \mathcal{H}$ be a closed subspace and $U,T \in B(\mathcal{H})$ be invertible operators such that $\pi_{UW}UT=U\pi_{W}.$ Then $U^{*}UW\subseteq \left( T^{*}\right) ^{-1}W.$
\end{lemma}

\begin{proof}
Suppose that $f \in W^\perp,$ then $\pi_{UW}UTf=U\pi_{W}f=0.$ Hence, for every $g \in W,$ we get
\begin{align*}
\left\langle f,T^{*}U^{*}Ug \right\rangle &= \left\langle UTf,Ug \right\rangle 
\\&=\left\langle \pi_{UW}UTf,Ug \right\rangle =0,
\end{align*}
which yields that $f \in \left( T^{*}U^{*}UW\right) ^\perp.$ Thus, it gives $U^{*}UW\subseteq \left( T^{*}\right) ^{-1}W.$
\end{proof}

It is worthy of note that the converse of the above lemma is not generally true. It is sufficient to take $ W=\mathcal{H}_{2}\times \{0\} \subset \mathcal{H}_{3},~U=I_{\mathcal{H}_{3}}$ and
\begin{equation*}
T=\begin{pmatrix}
1 & 1 & 0 \\
0 & 1 & 0 \\
0 & 0 & 1 \\
\end{pmatrix}.
\end{equation*}

\begin{lemma}\label{r.e}
Let $\mathcal{W}=\{(W_{i},\omega_{i})\}_{i \in I} $ be a fusion frame for $\mathcal {H}$ and $U \in B(\mathcal{H})$ an invertible operator. Then the following are equivalent:
\begin{itemize}
\item[$(i)$] $\mathcal{W}$ is a fusion Riesz basis.
\item[$(ii)$] $ \left( U^*\right) ^{-1}S_{\mathcal{W}}^{-1}W_{i}\perp UW_{j}$ for all $i,j \in I, i\neq j. $
\item[$(iii)$] $ \omega_{i}^{2}\pi_{W_{i}}S_{W}^{-1}\pi_{W_{j}}=\delta_{i,j}\pi_{W_{j}}~ for~all~ i , j \in I $.
\item[$(iv)$] $\omega_{i}^{2}\pi_{UW_{i}}\left( U^*\right) ^{-1}S_{\mathcal{W}}^{-1}\pi_{W_{j}}=\delta_{i,j}\pi_{UW_{j}}\left( U^*\right) ^{-1}$ for all $i,j \in I.$
\end{itemize}
\end{lemma}

\begin{proof}
$(i) \Leftrightarrow (iii)$ was proved in \cite[Proposition 2.4]{mitra sh.}. Also, $(i) \Leftrightarrow (ii)$ is obvious.
 
$(iii) \Rightarrow (iv)$ Applying Proposition \ref{Gavruta1} gives 
\begin{align*}
\omega_{i}^{2}\pi_{UW_{i}}\left( U^*\right) ^{-1}S_{\mathcal{W}}^{-1}\pi_{W_{j}}&=\omega_{i}^{2}\pi_{UW_{i}}\left( U^*\right) ^{-1}\pi_{W_{i}}S_{\mathcal{W}}^{-1}\pi_{W_{j}}
\\&=\delta_{i,j}\pi_{UW_{i}}\left( U^*\right) ^{-1}\pi_{W_{j}}
\\&=\delta_{i,j}\pi_{UW_{j}}\left( U^*\right) ^{-1},
\end{align*}
for all $i,j \in I.$

$(iv) \Rightarrow (ii)$ Let $f \in W_{i}$ and $g \in W_{j},$ where $i\neq j$. Then
\begin{align*}
\left\langle \left( U^*\right) ^{-1}S_{\mathcal{W}}^{-1}f,Ug \right\rangle &=\left\langle \left( U^*\right) ^{-1}S_{\mathcal{W}}^{-1}\pi_{W_{i}}f,\pi_{UW_{j}}Ug \right\rangle 
\\&=\left\langle \pi_{UW_{j}}\left( U^*\right) ^{-1}S_{\mathcal{W}}^{-1}\pi_{W_{i}}f,Ug \right\rangle =0.
\end{align*}
\end{proof}

\subsection{Operator-scalable fusion frames}
In what follows, we state a general definition of the operator-scalability for fusion frames. Then we study the results that are comparable to those observed for scalable ordinary frames.

\begin{definition}
A fusion frame $\mathcal{W}=\{(W_{i},\omega_{i})\}_{i \in I} $ of $\mathcal {H}$ is called \textit{operator-scalable} if there exists an invertible operator $U\in{B(\mathcal{H})}$ and a family of weights $ {\gamma}:=\{\gamma_{i}\}_{i \in I} $ such that $ U\mathcal{W}_{\gamma}:=\{(UW_{i},\omega_{i}\gamma_{i})\}_{i \in I} $ is a Parseval fusion frame for $\mathcal {H}$.
\end{definition}

It is worth noting that if $\mathcal{W}=\{(W_{i},\omega_{i})\}_{i \in I} $ is a fusion frame for $\mathcal {H}$ and $U\in{B(\mathcal{H})}$ is an invertible operator, then  $U\mathcal{W}=\{(UW_{i},\omega_{i})\}_{i \in I}$ is also a fusion frame by Theorem 2.4 of \cite{16}. Henceforth, we use $(U,\gamma)$-scalable to denote $\mathcal{W}$ is operator-scalable with the invertible operator $U\in{B(\mathcal{H})}$ and the family of weights $\gamma=\{\gamma_{i}\}_{i \in I},$ and simply $U$-scalable when the weights need not be explicitly mentioned. The general relationship between the operator-scalability and the weight-scalability of $\mathcal{W}$ can be easily determined from the following equivalence: 
\begin{equation}\label{ooo}
\mathcal{W} \textit{\ is \ } (U,\gamma)\textit{-scalable}~ \Leftrightarrow ~U\mathcal{W} \textit{\ is \ } \gamma\textit{-scalable}.
\end{equation}
In particular, if $U$ satisfies the condition $U^*UW_{i}\subseteq W_{i}$ for $i \in I,$ then a simple computation indicates that $\mathcal{W}$ is $(U,\gamma)$-scalable if and only if $\mathcal{W}$ is $\gamma$-scalable. Indeed, by applying Proposition \ref{Gavruta1}, we have
\begin{equation*}
S_{U\mathcal{W}_{\gamma}}=\sum_{i \in I}\left( \omega_{i}\gamma_{i}\right) ^2\pi_{UW_{i}}=US_{\mathcal{W}_{\gamma}}U^{-1}.
\end{equation*}

In order to obtain a characterization of operator-scalable fusion frames, we define the operator $D_{U_{\gamma}}$ corresponding to an operator $U$ and a family of weights $\gamma=\{\gamma_{i}\}_{i \in I}$ by
\begin{align*}
D_{U_{\gamma}}:&\sum_{i \in I}\bigoplus W_{i} \rightarrow \sum_{i \in I}\bigoplus UW_{i}
\\&\{f_{i}\}_{i\in I}\mapsto \left\lbrace \gamma_{i}^{-1}Uf_{i}\right\rbrace _{i\in I},
\end{align*}
where $f_{i}\in W_{i}$ for $i\in I$. Then, $D_{U_{\gamma}}$ is a bounded invertible operator provided that $\gamma^{-1} \in \ell^{\infty}.$ The following theorem extends Proposition 2.4 of \cite{Gitta 13} to the context of fusion frames.
\begin{theorem}\label{frame operator D}
Let $\mathcal{W}=\{(W_{i},\omega_{i})\}_{i \in I} $ be a fusion frame for $\mathcal {H}$ and $U\in B(\mathcal{H})$ be an invertible operator. Let $\gamma^{-1} \in \ell^{\infty}.$ Then the following are equivalent:
\begin{itemize}
\item[$(i)$] $\mathcal{W}$ is $(U,\gamma)$-scalable, 
\item[$(ii)$] The invertible operator $D_{U_{\gamma}}$ on $\sum_{i \in I}\bigoplus W_{i}$ satisfies
\begin{equation}\label{DU}
T_{\mathcal{W}}D_{U_{\gamma}}^{-1}\left( D_{U_{\gamma}}^{*}\right) ^{-1}T_{\mathcal{W}}^{*}=(U^{*}U)^{-1}.
\end{equation}
\end{itemize}
\end{theorem}

\begin{proof}
Suppose that $\{f_{i}\}_{i\in I} \in \sum_{i \in I}\bigoplus W_{i}$  and $\{g_{i}\}_{i\in I}\in \sum_{i \in I}\bigoplus UW_{i}$. Using the definition of $D_{U_{\gamma}}$ we get
\begin{align*}
\left\langle  \left( D_{U_{\gamma}}^{*}\right) ^{-1}\left\lbrace f_{i}\right\rbrace _{i\in I},\left\lbrace g_{i}\right\rbrace _{i\in I}\right\rangle  &= \left\langle \left\lbrace f_{i}\right\rbrace _{i\in I},\left\lbrace \gamma_{i}{U}^{-1}g_{i}\right\rbrace _{i\in I}\right\rangle 
\\&=\sum_{i \in I}\left\langle  \gamma_{i}\left( {U}^{*}\right) ^{-1}f_{i},\pi_{UW_{i}}g_{i}\right\rangle 
\\&= \left\langle \left\lbrace \gamma_{i}\pi_{UW_{i}}\left( {U}^{*}\right) ^{-1}f_{i}\right\rbrace _{i\in I},\left\lbrace g_{i}\right\rbrace _{i\in I}\right\rangle  .
\end{align*}
According to Proposition \ref{Gavruta1}, for every $f \in \mathcal {H}$, the frame operator $S_{U\mathcal{W}_{\gamma}}$ is obtained in terms of the operator $D_{U_{\gamma}}$ as follows;
\begin{align*}
UT_{\mathcal{W}}D_{U_{\gamma}}^{-1}\left( D_{U_{\gamma}}^{*}\right) ^{-1}T_{\mathcal{W}}^{*}U^{*}f&=UT_{\mathcal{W}}D_{U_{\gamma}}^{-1}\left( D_{U_{\gamma}}^{*}\right) ^{-1}\left\lbrace \omega_{i}\pi_{W_{i}}U^{*}f\right\rbrace _{i\in I}
\\&=UT_{\mathcal{W}}\left\lbrace \gamma_{i}^{2}\omega_{i}U^{-1}\pi_{UW_{i}}\left( {U}^{*}\right) ^{-1}\pi_{W_{i}}U^{*}f\right\rbrace _{i\in I}
\\&=UT_{\mathcal{W}}\left\lbrace \gamma_{i}^{2}\omega_{i}U^{-1}\pi_{UW_{i}}f\right\rbrace _{i\in I}
\\&=\sum_{i \in I}\left( \omega_{i}\gamma_{i}\right) ^{2}\pi_{UW_{i}}f=S_{U\mathcal{W}_{\gamma}}f.
\end{align*}
Thus,
\begin{equation*}
T_{\mathcal{W}}D_{U_{\gamma}}^{-1}\left( D_{U_{\gamma}}^{*}\right) ^{-1}T_{\mathcal{W}}^{*}=U^{-1}S_{U\mathcal{W}_{\gamma}}\left( U^{*}\right) ^{-1},
\end{equation*}
which completes the proof.
\end{proof}

The scalability in ordinary frames is invariant under unitary operators with the same coefficients. However, this fact is not applicable to the fusion frame setting. In fact, if $K\in B(\mathcal{H})$ is a unitary operator and $\mathcal{W}$ is a $(T,\gamma)$-scalable fusion frame, then $K\mathcal{W}$ is not necessarily $(T,\gamma)$-scalable, see Example \ref{e.g Riesz}. 

\subsection{Fusion Riesz bases}
In this subsection, we provide a characterization for the operator-scalability of fusion Riesz bases. First, assume that $\mathcal{W}=\{(W_{i},\omega_{i})\}_{i \in I}$ is a fusion Riesz basis for $\mathcal {H}$ and $\{f_{i,j}\}_{j \in J_{i}}$ a Riesz basis for $W_{i},$ for all $i \in I.$ It is well known that $\mathcal {F}=\{\omega_{i}f_{i,j}\}_{i \in I,j \in J_{i}}$ is a Riesz basis for $\mathcal {H}$ and $S_{\mathcal{W}}=S_{\mathcal{F}}.$ Hence, there exists an orthonormal basis $\{e_{i,j}\}_{i \in I,j \in J_{i}}$ for $\mathcal {H}$ such that $T_{\mathcal{F}}e_{i,j}=\omega_{i}f_{i,j},$ where $T_{\mathcal{F}}$ is the synthesis operator of $\mathcal{F}$. Thus, by putting $U_{0}:=T_{\mathcal{F}}^{-1},$ it is easily deduced that $\left\lbrace U_{0}W_{i}\right\rbrace _{i \in I}$ is an orthogonal fusion frame, and so $S_{U_{0}\mathcal{W}}=I_{\mathcal {H}}.$ 
Furthermore, for every $f \in \mathcal {H}$ we get
\begin{equation*}
S_{\mathcal{W}}f=S_{\mathcal{F}}f=T_{\mathcal{F}}T_{\mathcal{F}}^{*}f=\left( U_{0}^{*}U_{0}\right) ^{-1}f.
\end{equation*}
Hence, $S_{\mathcal{W}}^{-1}=U_{0}^*U_{0}.$ This result motivates us to describe all operators $U$ that turn a fusion Riesz basis into a $U$-scalable one.

\begin{theorem}\label{in case Riesz}
Let $\mathcal{W}=\{(W_{i},\omega_{i})\}_{i \in I} $ be a fusion Riesz basis for $\mathcal {H}$. Then the following are equivalent:
\begin{itemize}
\item[$(i)$] $\mathcal{W}$ is $\left( U,\omega^{-1}\right) $-scalable.
\item[$(ii)$] $\mathcal{W}$ is $\left( \left( S_{\mathcal{W}}U^*\right) ^{-1},\omega^{-1}\right) $-scalable.
\item[$(iii)$] $U^*UW_{i}=S_{\mathcal{W}}^{-1}W_{i}$ for all $ i \in I.$   
\item[$(iv)$] $\{(UW_{i},\omega_{i})\}_{i \in I} $ is an orthogonal fusion Riesz basis.
\end{itemize}
\end{theorem}

\begin{proof}
$(i)\Leftrightarrow (iii)$ Let $\mathcal{W}$ be $\left( U,\omega^{-1}\right) $-scalable. Using the fact that
\begin{equation}\label{tyt}
UW_{i}\perp \left( U^*\right) ^{-1}S_{\mathcal{W}}^{-1}W_{j},\quad (i\neq j),
\end{equation}
we get
\begin{equation}\label{sigma2}
\begin{aligned}
\left( U^*\right) ^{-1}S_{\mathcal{W}}^{-1}\pi_{W_{i}}&=S_{U\mathcal{W}_{\omega^{-1}}}\left( U^*\right) ^{-1}S_{\mathcal{W}}^{-1}\pi_{W_{i}}
\\&=\sum_{j \in I}\pi_{UW_{j}}\left( U^*\right) ^{-1}S_{\mathcal{W}}^{-1}\pi_{W_{i}}
\\&=\pi_{UW_{i}}\left( U^*\right) ^{-1}S_{\mathcal{W}}^{-1}\pi_{W_{i}}, \quad (i \in I).
\end{aligned}
\end{equation}
Hence, $ \left( U^*\right) ^{-1}S_{\mathcal{W}}^{-1}W_{i}\subseteq UW_{i} ,$ which implies that $S_{\mathcal{W}}^{-1}W_{i}\subseteq U^*UW_{i},$ for all $i \in I$.
On the other hand, by applying \eqref{sigma2} and Lemma \ref{r.e}$(iv)$, we get
\begin{equation*}
\left( U^*\right) ^{-1}S_{\mathcal{W}}^{-1}\pi_{W_{i}}=\pi_{UW_{i}}\left( U^*\right) ^{-1}S_{\mathcal{W}}^{-1}\pi_{W_{i}}=\pi_{UW_{i}}\left( U^*\right) ^{-1}.
\end{equation*}
So, for each $f,g \in \mathcal {H},$ we induce
\begin{align*}
\left\langle  \pi_{UW_{i}}US_{\mathcal{W}}f,g \right\rangle  &= \left\langle f , S_{\mathcal{W}}U^{*}\pi_{UW_{i}}g \right\rangle 
\\&= \left\langle f , S_{\mathcal{W}}U^{*}\pi_{UW_{i}}\left( U^*\right) ^{-1}U^*g \right\rangle 
\\&= \left\langle f , \pi_{W_{i}}U^*g \right\rangle  = \left\langle  U\pi_{W_{i}}f,g \right\rangle.
\end{align*}
Thus, $\pi_{UW_{i}}US_{\mathcal{W}}=U\pi_{W_{i}},$ for all $i \in I.$ It follows from Lemma \ref{extend Gav} that $U^*UW_{i}\subseteq S_{\mathcal{W}}^{-1}W_{i}.$ Therefore, $U^*UW_{i}=S_{\mathcal{W}}^{-1}W_{i},$ for all $i \in I$. For the converse implication, since $UW_{i}=\left( U^*\right) ^{-1}S_{\mathcal{W}}^{-1}W_{i},$ then by employing Lemma \ref{r.e}, \eqref{tyt} and taking $ \gamma_{i}=\omega_{i}^{-1},$ we conclude that
\begin{align*}
S_{U\mathcal{W}_{\gamma}}\left( U^*\right) ^{-1}&=\sum_{i \in I}\pi_{UW_{i}}\left( U^*\right) ^{-1}
\\&=\sum_{i \in I}\omega_{i}^{2}\pi_{UW_{i}}\left( U^*\right) ^{-1}S_{\mathcal{W}}^{-1}\pi_{W_{i}}
\\&=\sum_{i \in I}\omega_{i}^{2}\pi_{\left( U^*\right) ^{-1}S_{\mathcal{W}}^{-1}W_{i}}\left( U^*\right) ^{-1}S_{\mathcal{W}}^{-1}\pi_{W_{i}}
\\&=\left( U^*\right) ^{-1}\sum_{i \in I}\omega_{i}^{2}S_{\mathcal{W}}^{-1}\pi_{W_{i}}=\left( U^*\right) ^{-1}.
\end{align*}
Hence, $S_{U\mathcal{W}_{\gamma}}=I_{\mathcal {H}}$ and consequently $\mathcal{W}$ is $\left( U,\omega^{-1}\right) $-scalable.

$(i)\Leftrightarrow (ii)$ Let $\mathcal{W}$ be $\left( U,\omega^{-1}\right) $-scalable. In view of $(iii)$, we get
\begin{align*}
\left( US_{\mathcal{W}}\right) ^{-1}\left( S_{\mathcal{W}}U^*\right) ^{-1}W_{i}&=S_{\mathcal{W}}^{-1}U^{-1}\left( U^*\right) ^{-1}S_{\mathcal{W}}^{-1}W_{i}
\\&=S_{\mathcal{W}}^{-1}U^{-1}UW_{i}
\\&=S_{\mathcal{W}}^{-1}W_{i},
\end{align*}
for each $ i\in I $. Therefore, $\mathcal{W}$ is $\left( S_{\mathcal{W}}U^*\right) ^{-1}$-scalable by using $(iii)\Rightarrow(i)$. The converse is derived by applying a symmetry argument.

$(iii)\Rightarrow (iv)$ Applying $(iii)$ leads to 
\begin{equation*}
\left\langle UW_{i},UW_{j} \right\rangle = \left\langle U^{*}UW_{i},W_{j} \right\rangle = \left\langle S_{\mathcal{W}}^{-1}W_{i},W_{j} \right\rangle =0,
\end{equation*}
for each $i\neq j.$ Thus, $(iv)$ holds.

$(iv)\Rightarrow (i)$ Taking $ \gamma_{i}=\omega_{i}^{-1}$, we obtain
\begin{equation*}
S_{U\mathcal{W}_{\gamma}}=\sum_{i \in I} \pi_{UW_{i}}=\pi_{\oplus_{i \in I} UW_{i}}=I_{\mathcal {H}},
\end{equation*}
which yields $(i)$ is satisfied.
\end{proof}

As an immediate result of the above theorem, all fusion Riesz bases are $(U,\gamma)$-scalable by taking $U=S_{\mathcal{W}}^{-1/2}$ and $\gamma=\omega^{-1},$ see also \cite[Proposition 4.7]{Excess 1}.

\begin{corollary}\label{orth. case}
If $\mathcal{W}$ is an orthogonal fusion Riesz basis, then the conditions (i)-(iv) in Theorem \ref{in case Riesz} are equivalent to the following conditions:
\begin{itemize}
\item[$(v)$] $U^*UW_{i}=W_{i}$ for all $ i \in I.$ 
\item[$(vi)$] $\mathcal{W}$ is $\left( \left( U^*\right) ^{-1},\omega^{-1}\right) $-scalable.
\item[$(vii)$] $\widetilde{\mathcal{W}}$ is $\left( U,\omega^{-1}\right) $-scalable.
\end{itemize}
\end{corollary}

\begin{proof}
Since $\mathcal{W}$ is orthogonal, then 
\begin{equation}\label{vvvv}
S_{\mathcal{W}}^{-1}\pi_{W_{i}}=\omega_{i}^{-2}\pi_{W_{i}}, \quad (i \in I),
\end{equation}
by \cite[Proposition 3.3]{scaling weights}. Thus, $S_{\mathcal{W}}^{-1}W_{i}=W_{i}$ for all $ i \in I,$ which implies that $(iii)\Leftrightarrow (v).$ The equivalence $(i)\Leftrightarrow (vi)$ is clear by $(v).$ Moreover, it yields from \eqref{vvvv} that
\begin{equation*}
\left\langle  US_{\mathcal{W}}^{-1}f,US_{\mathcal{W}}^{-1}g\right\rangle  =\left( \omega_{i}\omega_{j}\right) ^{-2}\left\langle  Uf,Ug \right\rangle  ,
\end{equation*}
for each $f \in W_{i}$ and $g \in W_{j}$ $(i\neq j).$ Therefore, $(i)\Leftrightarrow (vii)$ directly follows from $(iv).$
\end{proof}

In the following, we present two examples which confirm our results.

\begin{example}\label{e.g Riesz}
(1) Consider fusion Riesz basis $\mathcal{W}=\{(W_{i},\omega_{i})\}_{i=1}^{2},$ defined as $W_{1}=\textnormal{span}\{(1,1,0)\}$ and $W_{2}=\{0\}\times \Bbb{R}^{2}$ for $\Bbb{R}^{3}$. Let
\begin{align*}
U=\begin{pmatrix}
   1  & 0 & 0\\
   1  & -1  & 0 \\
   1  & -1  & 1 
\end{pmatrix}.
\end{align*}
Due to $UW_{1}\perp UW_{2}$, Theorem \ref{in case Riesz} verifies that $\mathcal{W}$ is $\left( U,\omega^{-1}\right) $-scalable. More generally, $\mathcal{W}$ is $\left( T,\omega^{-1}\right) $-scalable by every invertible operator $T$ of the form 
\begin{align*}
T=\begin{pmatrix}
   a_{1}  & 0 & 0\\
   a_{2}  & -a_{2}  & a_{3} \\
   a_{4}  & -a_{4}  & a_{5} 
\end{pmatrix},
\end{align*}
where $ a_{i}\in \Bbb{R}$ for all $1 \leq i \leq 5.$ In addition, if we take 
\begin{align*}
K=\frac{1}{\sqrt{2}}\begin{pmatrix}
   1  & 0 & 1\\
   0  & \sqrt{2}   & 0 \\
   -1 & 0  & 1
\end{pmatrix},
\end{align*}
then $K$ is a unitary operator and it is easily seen that $K\mathcal{W}$ is not $\left( T,\omega^{-1}\right) $-scalable.

(2) Consider $ W_{1}=\Bbb{R}^{2}\times \{0\} $ and $ W_{2}=\{0\} \times \{0\}\times \Bbb{R}$. Then $\mathcal{W}=\{(W_{i},\omega_{i})\}_{i=1}^{2} $ is an orthogonal fusion Riesz basis for $\Bbb{R}^{3}.$ Take
\begin{align*}
U=\begin{pmatrix}
 ~~M & \begin{matrix} 0 \\[1ex] 0 \end{matrix} \\
 \begin{matrix} 0 & 0 \end{matrix} & c \\
\end{pmatrix},
\end{align*}
where $M$ is a $2\times 2$ invertible matrix and $c \in \mathbb{R}\smallsetminus \{0\}.$ One can simply observe that $U^*UW_{i}=W_{i}$ for $ i=1,2.$ Thus, $\mathcal{W}$ is $\left( U,\omega^{-1}\right) $-scalable, by Corollary \ref{orth. case}.
\end{example}

Subsequently, we examine several results derived from Theorem \ref{in case Riesz}.

\begin{corollary}\label{Riesz other}
Let $\mathcal{W}=\{(W_{i},\omega_{i})\}_{i \in I} $ be a fusion Riesz basis for $\mathcal {H}$ and $T\in B(\mathcal{H})$ be an invertible operator. Then the following are equivalent:
\begin{itemize}
\item[$(i)$] $\mathcal{W}$ is $\left( U,\omega^{-1}\right) $-scalable.
\item[$(ii)$] $T\mathcal{W}$ is $\left( UT^{-1},\omega^{-1}\right) $-scalable.
\item[$(iii)$] $\widetilde{\mathcal{W}}$ is $\left( US_{\mathcal{W}},\omega^{-1}\right) $-scalable.
\item[$(iv)$] $\widetilde{\mathcal{W}}$ is $\left( \left( U^{*}\right) ^{-1},\omega^{-1}\right) $-scalable.
\item[$(v)$] $U^*UW_{i}=T^{*}S_{T\mathcal{W}}^{-1}TW_{i}$ for all $ i \in I.$   
\end{itemize}
\end{corollary}

\begin{proof}
$(i)\Leftrightarrow (ii)$ is trivial.

$(i)\Leftrightarrow (iii)$ is obtained by putting $T=S_{\mathcal{W}}^{-1}$ in $(ii)$.

$(ii)\Leftrightarrow (v)$ Assume that $T\mathcal{W}$ is $\left( UT^{-1},\omega^{-1}\right) $-scalable. Since $T\mathcal{W}$ is also a fusion Riesz basis, then it is deduced from Theorem \ref{in case Riesz}$(iii)$ that
\begin{align*}
T^{*}S_{T\mathcal{W}}^{-1}TW_{i}&=T^{*}\left( UT^{-1}\right) ^{*}\left( UT^{-1}\right) TW_{i}
\\&=T^{*}\left( T^{*}\right) ^{-1}U^{*}UW_{i}
\\&=U^{*}UW_{i},
\end{align*}
for all $i \in I.$ The converse is demonstrated by means of a symmetry argument.

$(i)\Leftrightarrow (iv)$ Applying Proposition \ref{mitra sh.} we obtain
\begin{align*}
S_{\mathcal{W}}S_{\widetilde{\mathcal{W}}}\pi_{W_{i}}&=\sum_{j \in I}\omega_{j}^{2} \pi_{W_{j}}\sum_{k \in I}\omega_{k}^{2}\pi_{S_{\mathcal{W}}^{-1}W_{k}}\pi_{W_{i}}
\\&=\sum_{j \in I}\omega_{j}^{4}\pi_{W_{j}}\pi_{S_{\mathcal{W}}^{-1}W_{j}}\pi_{W_{i}}
\\&=\omega_{i}^{2}\pi_{W_{i}}.
\end{align*}
for each $i \in I.$ Thus, $S_{\mathcal{W}}S_{\widetilde{\mathcal{W}}}W_{i}=W_{i}$ and hence,
\begin{equation}\label{333}
\begin{aligned}
S_{\widetilde{\mathcal{W}}}^{-1}\widetilde{W}_{i}&=S_{\widetilde{\mathcal{W}}}^{-1}S_{\mathcal{W}}^{-1}W_{i}
\\&=\left( S_{\mathcal{W}}S_{\widetilde{\mathcal{W}}}\right) ^{-1}W_{i} 
\\&=W_{i},
\end{aligned}
\end{equation}
for all $i \in I$. If $\mathcal{W}$ is $\left( U,\omega^{-1}\right) $-scalable, then by applying Theorem \ref{in case Riesz} and \eqref{333} we get 
\begin{align*}
U^{-1}\left( U^{*}\right) ^{-1}\widetilde{W_{i}}&=\left( U^{*}U\right) ^{-1}S_{\mathcal{W}}^{-1}W_{i}
\\&=W_{i} 
\\&=S_{\widetilde{\mathcal{W}}}^{-1}\widetilde{W_{i}},
\end{align*}
for $i \in I$. Therefore, $\widetilde{\mathcal{W}}$ is $\left( \left( U^{*}\right) ^{-1},\omega^{-1}\right)$-scalable, by Theorem \ref{in case Riesz}. Conversely, if $\widetilde{\mathcal{W}}$ has this property, according to \eqref{333} and Theorem \ref{in case Riesz}, we induce
\begin{align*}
U^{*}UW_{i}&=U^{*}US_{\widetilde{\mathcal{W}}}^{-1}\widetilde{W_{i}} 
\\&=U^{*}U U^{-1}\left( U^{*}\right) ^{-1}\widetilde{W_{i}} 
\\&=S_{\mathcal{W}}^{-1}{W_{i}},
\end{align*}
for each $i \in I$, which ensures that $\mathcal{W}$ is $\left( U,\omega^{-1}\right) $-scalable. 
\end{proof}

\begin{example}
Consider the fusion Riesz basis $\mathcal{W}=\{(W_{i},\omega_{i})\}_{i=1}^2 $ presented in Example \ref{e.g Riesz}(1). A straightforward calculation shows that
 \begin{align*}
S_{\mathcal{W}}=\frac{1}{2}\begin{pmatrix}
   \omega_{1}^2  & \omega_{1}^2 & 0\\
   \omega_{1}^2  & \omega_{1}^2+2\omega_{2}^2 & 0 \\
   0  & 0  & 2\omega_{2}^2 
\end{pmatrix},
\end{align*}
and the subspaces
\begin{equation*}
\widetilde{W}_{1}=\textnormal{span} \left\lbrace (1,0,0)\right\rbrace ,~ \widetilde{W}_{2}=\textnormal{span} \left\lbrace (0,0,1) , (-1,1,0)\right\rbrace ,
\end{equation*}
with the weights $\omega= \{\omega_{i}\}_{i \in I} $ is the canonical dual of $\mathcal{W}$. As previously stated, $\mathcal{W}$ is $\left( T,\omega^{-1}\right) $-scalable for all invertible operators $ T $ of the form 
\begin{align*}
T=\begin{pmatrix}
   a_{1}  & 0 & 0\\
   a_{2}  & -a_{2}  & a_{3} \\
   a_{4}  & -a_{4}  & a_{5} 
\end{pmatrix},
\end{align*}
where $ a_{i}\in \Bbb{R}$ for $1 \leq i \leq 5.$ A simple computation gives 
\begin{align*}
\left( T^{*}\right) ^{-1}\widetilde{W}_{1}&=\textnormal{span} \left\lbrace ( a_3 a_4-a_2 a_5 ,0,0) \right\rbrace , \\
\left( T^{*}\right) ^{-1}\widetilde{W}_{2}&=\textnormal{span} \left\lbrace (0,a_4,-a_2) , (0,a_5,-a_3)\right\rbrace .
\end{align*}
Clearly, $ \left( T^{*}\right) ^{-1}\widetilde{W}_{1}\perp \left( T^{*}\right) ^{-1}\widetilde{W}_{2}.$ Therefore, it follows form Theorem \ref{in case Riesz} that $ \widetilde{\mathcal{W}}$ is $\left( \left( T^{*}\right) ^{-1},\omega^{-1}\right) $-scalable, which confirms the validity of Corollary \ref{Riesz other}$(iv)$.
\end{example}

Now, we focus on the scalability of some duals of fusion Riesz bases. It should be noted that, unlike ordinary frames, the duality does not possess the symmetric property in fusion frame setting \cite{Osgooei}. Furthermore, the excess of a fusion frame and its dual may not be the same, see \cite{excess of fusion} for some examples. Given a fusion Riesz basis, we indicate that any of its 1-excess dual fusion frames cannot be operator-scalable. 

\begin{corollary}
Every 1-excess dual fusion frame of a fusion Riesz basis is not operator-scalable.
\end{corollary}

\begin{proof}
Let $\mathcal{V}=\{V_{i}\}_{i \in I} $ be a 1-excess dual fusion frame of a fusion Riesz basis $\mathcal{W}=\{(W_{i},\omega_{i})\}_{i \in I} .$ Then $S_{\mathcal{W}}^{-1}W_{i}\subseteq V_{i}$ for all $i \in I,$ by \cite[Corollary 2.6]{Arabyani dual}. Since $e(\mathcal{V})=1,$ there exists a unique $j \in I$ such that 
\begin{align*}
V_{i}= \begin{cases}
S_{\mathcal{W}}^{-1}W_{i}, \quad  i \neq j, 
\\ S_{\mathcal{W}}^{-1}W_{j} \cup \textnormal{span}\{x\}, \quad  i = j, 
\end{cases}
\end{align*}
for some $x \in \mathcal {H}.$ This implies that the subspace containing the excess element $x$ has dimension greater than one. As such, $\mathcal{V}$ is not weight-scalable and consequently is not operator-scalable, by \cite[Theorem 3.6]{scaling weights} and \eqref{ooo}.
\end{proof}

\begin{Remark}
Let $\mathcal{W}=\{(W_{i},\omega_{i})\}_{i \in I} $ be a fusion Riesz basis for $\mathcal {H}$ and $\{e_{i,j}\}_{j \in J_{i}}$ be an orthonormal basis for $W_{i}~ (i\in I)$. Then $\mathcal {F}=\{\omega_{i}e_{i,j}\}_{i \in I,j \in J_{i}}$ and $\widetilde{\mathcal {F}}=\left\lbrace \omega_{i}S_{\mathcal{W}}^{-1}e_{i,j}\right\rbrace _{i \in I,j \in J_{i}}$ are Riesz bases for $\mathcal{H}$. Consider 
\begin{equation*}
\mathcal{A}={\left(\begin{matrix}
M_{1} &  0   & \cdots & 0 \\
 0  &  M_{2}  &  \cdots & 0 \\
 \vdots & \vdots &   \ddots &  \vdots \\
 0 & 0  & \cdots & M_{|I|}
\end{matrix}\right)}_{\sum_{i \in I}|J_{i}|\times |I|},
\end{equation*}
where $M_{i}=\left[ c_{i,1}, \ldots ,c_{i,|J_{i}|} \right] ^t$ for $i \in I$ and $c_{i,k}$ is a non zero scalar for any $k \in J_{i}$. If $U$ is an invertible operator on $\mathcal{H}$ such that
\begin{equation*}
\left[U^*U\right]_{\mathcal {F},\widetilde{\mathcal {F}}}=\mathcal{A},
\end{equation*}
then it yields that $U^*U$ defines a bounded operator on $\mathcal{H}$ as follows;
\begin{align*}
U^*U:& W_{i} \rightarrow S_{\mathcal{W}}^{-1}W_{i}
\\&e_{i,j}\mapsto \sum_{k \in J_{i}}c_{i,k}S_{\mathcal{W}}^{-1}e_{i,k}.
\end{align*}
That means that $U^*UW_{i}=S_{\mathcal{W}}^{-1}W_{i}$ for all $i \in I$ and thereby $\mathcal{W}$ is $(U,\gamma)$-scalable, by Theorem \ref{in case Riesz}. Hence, the matrix representation of all operators $U$ that make $\mathcal{W}$ scalable is given by
\begin{equation*}
\left[ U \right]_{\mathcal {F},\widetilde{\mathcal {F}}} =E\mathcal{A}^{\frac{1}{2}},
\end{equation*}
where $E$ is a partial isometry operator on $\mathcal{H}$ \cite{Furuta}. In light of the fact that $S_{\mathcal{W}}=S_{\mathcal{F}},$ we get
\begin{align*}
\Vert Ue_{i,j} \Vert^{2}&=\left\langle U^*Ue_{i,j},e_{i,j} \right\rangle  
\\&=\left\langle \sum_{k \in J_{i}}c_{i,k}S_{\mathcal{W}}^{-1}e_{i,k},e_{i,j}\right\rangle 
\\&=\sum_{k \in J_{i}}c_{i,k}\left\langle S_{\mathcal{F}}^{-1}e_{i,k},e_{i,j}\right\rangle 
\\&=c_{i,j}\left\langle S_{\mathcal{F}}^{-1}e_{i,j},e_{i,j}\right\rangle =c_{i,j},
\end{align*}
for all $i \in I, j \in J_{i}.$
\end{Remark}

\section{Operator-Scalable 1-excess Fusion Frames}

This section is dedicated to the study of the operator-scalability of 1-excess fusion frames. It is known that if $\mathcal{F}$ is an ordinary frame for $\mathcal {H}$ with the frame operator $S_{\mathcal{F}},$ then $S_{\mathcal{F}}^{-1/2}\mathcal{F}$ is a Parseval frame. However, this is not necessarily the case in the context of fusion frames. In other words, if $\mathcal{V}$ is an arbitrary overcomplete fusion frame for $\mathcal {H},$ then $\mathcal{V}$ is not necessarily $S_{\mathcal{V}}^{-1/2}$-scalable, not even by changing the weights \cite{Excess 1}. Motivated by this result, we are looking for operators that ensure the scalability of 1-excess fusion frames. In the sequel, without losing the generality, we may take into account 1-uniform fusion frames. Nevertheless, it is important to note that our results are valid for 1-excess fusion frames with an arbitrary family of weights. 

In \cite{excess of fusion}, it was shown that every fusion frame for $\mathcal {H}$ has the same excess with every its local frame generated by Riesz bases. In this regard, it is noteworthy that several well known results in ordinary frames cannot be generalized to the fusion frame setting. For instance, there exist fusion frames that are not operator-scalable by any invertible operator and any sequence of weights. It suffices to consider the 1-excess fusion frame $\mathcal{V}$ constituted by 
\begin{equation*}
V_{1}=\textnormal{span}\left\lbrace e_{1},e_{2}\right\rbrace ,~V_{2}=\textnormal{span}\left\lbrace e_{1},e_{3}\right\rbrace ,~V_{3}=\textnormal{span}\left\lbrace e_{4}\right\rbrace ,
\end{equation*}
in which $\{e_{i}\}_{i=1}^4$ is the canonical orthonormal basis for $\Bbb{C}^{4}$, see \cite[Example 7.6]{Excess 1}.
However, our primary aim in this section is to characterize operator-scalable 1-excess fusion frames. Let $\mathcal{V}$ be a 1-excess fusion frame for $\mathcal {H}$. If the excess element belongs to a subspace of dimension greater than one, then $\mathcal{V}$ is not weight-scalable and so cannot be operator-scalable, as well, by \cite[Theorem 3.6]{scaling weights} and \eqref{ooo}. Hence, every operator-scalable 1-excess fusion frame in $\mathcal {H}$ is of the form
\begin{equation}\label{every 1 excess}
\mathcal{V}= V_{0} \cup \mathcal{W},
\end{equation}
where $\mathcal{W}=\{W_{i}\}_{i \in I}$ is a fusion Riesz basis, $V_{0}=\textnormal{span}\{x\}$ and $x =\sum_{i \in \sigma} x_{i}$ is a vector of $\mathcal {H}$ such that $x_{i} \in W_{i}$ and $\sigma=\left\lbrace i \in I\mid x_{i}\neq0 \right\rbrace $. Combining Corollary 3.8 of \cite{scaling weights} and \eqref{ooo}, we derive some necessary conditions for the operator-scalability of 1-excess fusion frames as follows.

\begin{corollary}\label{erer}
Let $\mathcal{V}$ be the 1-excess fusion frame given by \eqref{every 1 excess}. If $\mathcal{V}$ is $(U,\gamma)$-scalable, then the following statements hold: 
\begin{itemize}
\item[$(i)$] $\gamma_{0}< 1$ and $Ux$ is an eigenvector of $S_{U\mathcal{W}_\gamma}$ associated with eigenvalue $1-\gamma_{0}^2$.
\item[$(ii)$]  If $ j \in \sigma^{c},$ then $\gamma_{j}=1,~x \perp U^{*}UW_{j}$ and $U^{*}UW_{j}\perp W_{i}$ for $i\neq j.$ 
\item[$(iii)$]  If $ j\in \sigma,$ then $\gamma_{j}\neq 1,~x \not\perp U^{*}UW_{j}$ and $\textnormal{dim}UW_{j}=1$.
\end{itemize}
\end{corollary}

We now provide some necessary and sufficient conditions for a general 1-excess fusion frame to be operator-scalable. 

\begin{theorem}
Let $\mathcal{V}$ be the 1-excess fusion frame given by \eqref{every 1 excess}. Then the following statements hold:
\begin{itemize}
\item[$(i)$] $\mathcal{V}$ is $(U,\gamma)$-scalable if and only if 
\begin{equation}\label{Suw}
S_{\mathcal{W}}^{-1}\pi_{W_{i}} =\gamma_{0}^{2}U^{*}U\pi_{V_{0}}S_{\mathcal{W}}^{-1}\pi_{W_{i}}+\gamma_{i}^{2}U^{*}\pi_{UW_{i}}\left( U^{*}\right) ^{-1}, \quad (i \in I).
\end{equation}
\item[$(ii)$] If $\mathcal{V}$ is $(U,\gamma)$-scalable, then 
\begin{equation*}
S_{\mathcal{W}}^{-1}W_{i} \subseteq U^{*}U\left( W_{i}+\textnormal{span}\{x\}\right) , \quad (i \in I). 
\end{equation*} 
\end{itemize}
\end{theorem}

\begin{proof}
$(i)$ Suppose that $\mathcal{V}$ is $(U,\gamma)$-scalable, then there exists an invertible operator $U\in{B(\mathcal{H})}$ and a sequence of weights $\gamma=\{\gamma_{i}\}_{i \in I}$ such that 
\begin{equation}\label{ytyp}
f=S_{U\mathcal{V}_{\gamma}}f=\gamma_{0}^{2}\pi_{UV_{0}}f+S_{U\mathcal{W}_{\gamma}}f, \quad  (f\in \mathcal {H}).
\end{equation}
Without losing the generality, we may assume that $\Vert Ux \Vert=1.$ Replacing $f=\left( U^{*}\right) ^{-1}S_{\mathcal{W}}^{-1}\pi_{W_{i}}$ in \eqref{ytyp} and applying Lemma \ref{r.e}, we infer that 
\begin{align*}
S_{\mathcal{W}}^{-1}\pi_{W_{i}}&=U^{*}\left( U^{*}\right) ^{-1}S_{\mathcal{W}}^{-1}\pi_{W_{i}}
\\&=U^{*}\left(\gamma_{0}^{2}\pi_{UV_{0}}\left( U^{*}\right) ^{-1}S_{\mathcal{W}}^{-1}\pi_{W_{i}}+S_{U\mathcal{W}_{\gamma}}\left( U^{*}\right) ^{-1}S_{\mathcal{W}}^{-1}\pi_{W_{i}} \right) 
\\&=U^{*}\left(\gamma_{0}^{2}U\pi_{V_{0}}U^{*}\left( U^{*}\right) ^{-1}S_{\mathcal{W}}^{-1}\pi_{W_{i}}+\sum_{j=1}^{\infty}\gamma_{j}^{2}\pi_{UW_{j}}\left( U^{*}\right) ^{-1}S_{\mathcal{W}}^{-1}\pi_{W_{i}} \right) 
\\&=\gamma_{0}^{2}U^{*}U\pi_{V_{0}}S_{\mathcal{W}}^{-1}\pi_{W_{i}}+\gamma_{i}^{2}U^{*}\pi_{UW_{i}}\left( U^{*}\right) ^{-1},
\end{align*}
for each $i \in I,$ as required. Conversely, assume that \eqref{Suw} holds. Then we obtain
\begin{align*}
U^{*}S_{U\mathcal{V}_{\gamma}} \left( U^{*}\right) ^{-1}&=U^{*}\left(\gamma_{0}^{2}\pi_{UV_{0}}+S_{U\mathcal{W}_{\gamma}}\right) \left( U^{*}\right) ^{-1}
\\&=\gamma_{0}^{2}U^{*}\pi_{UV_{0}}\left( U^{*}\right) ^{-1}+U^{*}S_{U\mathcal{W}_{\gamma}}\left( U^{*}\right) ^{-1}
\\&=\gamma_{0}^{2}U^{*}U\pi_{V_{0}}+U^{*}S_{U\mathcal{W}_{\gamma}}\left( U^{*}\right) ^{-1}
\\&=\sum_{i \in I}\left(\gamma_{0}^{2}U^{*}U\pi_{V_{0}}S_{\mathcal{W}}^{-1}\pi_{W_{i}}+\gamma_{i}^{2}U^{*}\pi_{UW_{i}}\left( U^{*}\right) ^{-1} \right) 
\\&=\sum_{i \in I}S_{\mathcal{W}}^{-1}\pi_{W_{i}}=I_{\mathcal{H}}.
\end{align*}
Therefore, $S_{U\mathcal{V}_{\gamma}}=I_{\mathcal{H}},$ which implies that $\mathcal{V}$ is $(U,\gamma)$-scalable.

$(ii)$ Applying \eqref{Suw} gives that
\begin{align*}
S_{\mathcal{W}}^{-1}W_{i}=R\left( S_{\mathcal{W}}^{-1}\pi_{W_{i}}\right) &=R\left( \gamma_{0}^{2}U^{*}U\pi_{V_{0}}S_{\mathcal{W}}^{-1}\pi_{W_{i}}+\gamma_{i}^{2}U^{*}\pi_{UW_{i}}\left( U^{*}\right) ^{-1}\right) 
\\&\subseteq R\left(U^{*}U\pi_{V_{0}}S_{\mathcal{W}}^{-1}\pi_{W_{i}} \right) +R\left(U^{*}\pi_{UW_{i}}\left( U^{*}\right) ^{-1} \right)
\\&=\left\lbrace  \left\langle  S_{\mathcal{W}}^{-1}\pi_{W_{i}}f,x \right\rangle  U^{*}Ux : f \in \mathcal{H}  \right\rbrace  +R\left(U^{*}\pi_{UW_{i}}\right)
\\&=\textnormal{span}\{U^{*}Ux\}+U^{*}UW_{i}
\\&=U^{*}U\left( W_{i}+\textnormal{span}\{x\}\right),
\end{align*}
for all $i \in I,$ which completes the proof. 
\end{proof}

The next theorem identifies all operator-scalable 1-excess fusion frames.

\begin{theorem}\label{two equivalent}
Let $\mathcal{V}$ be the 1-excess fusion frame defined as in \eqref{every 1 excess} and $U \in B(\mathcal {H})$ be invertible. Then the following are equivalent:
\begin{itemize}
\item[$(i)$] $\mathcal{V}$ is $(U,\gamma)$-scalable,
\item[$(ii)$] There exists a partition $\{I_{1},I_{2}\}$ of $I$ such that 
\begin{equation*}
\mathcal{V}=\left\lbrace \textnormal{span}\left\lbrace f_{i} \right\rbrace \right\rbrace _{i \in I_{1}}\cup \left\lbrace W_{i} \right\rbrace _{i \in I_{2}},
\end{equation*}
$\textnormal{span}\{Uf_{i}\}_{i \in I_{1}}\oplus \textnormal{span}\{UW_{i}\}_{i \in I_{2}}=\mathcal {H}$ and one of the following holds:
\begin{itemize}
\item[$(a)$] $\{Uf_{i}\}_{i \in I_{1}}$ is a strictly scalable 1-excess frame sequence with coefficients $\{\gamma_{i}\}_{i \in I_{1}}$ and $\{UW_{i}\}_{i \in I_{2}}$ is an orthogonal sequence of subspaces.

\item[$(b)$] If $\mathcal{F}=\{f_{i}\}_{i \in I_{1}}$ and $\mathcal{W}'=\{W_{i}\}_{i \in I_{2}},$ then 
\begin{align*}
\begin{cases}
 U^{*}Uf=S_{\gamma \mathcal{F}}^{-1}f, \quad  f \in \textnormal{span}\{Uf_{i}\}_{i \in I_{1}}, 
\\U^{*}UW_{i}=S_{\mathcal{W}'}^{-1}W_{i}, \quad  i \in I_{2}.
\end{cases}
\end{align*}
\end{itemize}
\end{itemize}
\end{theorem}

\begin{proof}
$(i)\Rightarrow(a)$ If $\mathcal{V}$ is $(U,\gamma)$-scalable, then $U\mathcal{V}$ is $\gamma$-scalable, by \eqref{ooo}. Thus, according to \cite[Theorem 3.9]{scaling weights}, $U\mathcal{V}$  has a representation of the form
\begin{equation*}
U\mathcal{V}=\left\lbrace \textnormal{span}\left\lbrace Uf_{i} \right\rbrace \right\rbrace _{i \in I_{1}}\cup \left\lbrace UW_{i} \right\rbrace _{i \in I_{2}},
\end{equation*}
where $\{I_{1},I_{2}\}$ is a partition of $I$, ${\left\lbrace Uf_{i} \right\rbrace _{i \in I_{1}}}$ is a strictly scalable 1-excess frame sequence for $\mathcal{H}_{1}:=\textnormal{span}\{Uf_{i}\}_{i \in I_{1}}$ and $\left\lbrace UW_{i} \right\rbrace _{i \in I_{2}}$ is an orthogonal fusion Riesz basis for $\textnormal{span}\{UW_{i}\}_{i \in I_{2}}=\mathcal{H}_{1}^{\perp}.$

$(a)\Rightarrow(i)$ Due to the assumption, the 1-excess frame sequence $\{Uf_{i}\}_{i \in I_{1}}$ is strictly scalable by $\{\gamma_{i}\}_{i \in I_{1}},$ so
\begin{equation*}
\sum_{i \in I_{1}}\gamma_{i}^{2}\left\langle g,Uf_{i} \right\rangle Uf_{i}=g,
\end{equation*}
for all $g \in \mathcal{H}_{1}.$ Moreover, $\left\lbrace UW_{i} \right\rbrace _{i \in I_{2}}$ is a uniform scalable fusion Riesz basis for $\mathcal{H}_{1}^{\perp}$ i.e., $\sum_{i \in I_{2}}\pi_{UW_{i}}=I_{\mathcal{H}_{1}^{\perp}}.$ Put $\gamma_{i}=1$ for all $i \in I_{2}.$ Thus, by considering $\gamma=\{\gamma_{i}\}_{i \in I_{1}\cup I_{2}}$ and in view of the fact that for every $g \in \mathcal{H},$ there exist unique vectors $g_{1} \in \mathcal{H}_{1}$ and $g_{2} \in \mathcal{H}_{1}^{\perp}$ such that $g=g_{1}+g_{2}$ we obtain 
\begin{align*}
S_{U\mathcal{V}_{\gamma}}g&=\sum_{i \in I_{1}}\gamma_{i}^{2}\pi_{\textnormal{span}\{Uf_{i}\}}(g_{1}+g_{2})+\sum_{i \in I_{2}}\pi_{UW_{i}}(g_{1}+g_{2})
\\&=\sum_{i \in I_{1}}\gamma_{i}^{2}\left\langle g_{1},Uf_{i} \right\rangle Uf_{i}+\sum_{i \in I_{2}}\pi_{UW_{i}}g_{2}
\\&=g_{1}+g_{2}=g.
\end{align*}
Therefore, $\mathcal{V}$ is $(U,\gamma)$-scalable, by \eqref{ooo}.

In the end, it is enough to show that $(a)\Leftrightarrow(b)$. According to Theorem \ref{in case Riesz}, the orthogonality of $U\mathcal{W}'$ is equivalent to $U^{*}UW_{i}=S_{\mathcal{W}'}^{-1}W_{i}$ for all $i \in I_{2}$. In addition, if $U\mathcal{F}$ is strictly scalable with coefficients $\{\gamma_{i}\}_{i \in I_{1}},$ then for each $f \in \mathcal{H}_{1}$ we obtain 
\begin{align*}
f&=\sum_{i \in I_{1}}\gamma_{i}^{2}\left\langle f,Uf_{i} \right\rangle Uf_{i}
\\&=US_{\gamma \mathcal{F}}U^{*}f.
\end{align*}
It gives $U^{*}U=S_{\gamma \mathcal{F}}^{-1}$ on $\mathcal{H}_{1}.$ The converse is obtained by an analogous approach.
\end{proof}

The validity of the obtained results is illustrated through the following examples.

\begin{example}
Consider the fusion Riesz basis introduced in Example \ref{e.g Riesz}(1). As suggested by this example, $\mathcal{W}$ is $\left( T,\omega^{-1}\right) $-scalable by all operators $T$ of the form 
\begin{align*}
T=\begin{pmatrix}
   a_{1}  & 0 & 0\\
   a_{2}  & -a_{2}  & a_{3} \\
   a_{4}  & -a_{4}  & a_{5} 
\end{pmatrix},
\end{align*}
where $ a_{i}\in \Bbb{R}$ for all $1 \leq i \leq 5.$ Now, let $\mathcal{V}=\lbrace V_{1}, W_{2}\rbrace ,$ where $V_{1}=W_{1}+\textnormal{span}\{(0,a,b)\}$ and $ a,b \in \Bbb{R}$ are not simultaneously zero. Then $\mathcal{V}$ is a 1-excess fusion frame for $\mathcal {H}=\Bbb R^{3} $. As the excess element belongs to $V_{1}$ and $\textnormal{dim}UV_{1}=2$ for any invertible operator $U \in B(\mathcal {H})$, it thus follows from Corollary \ref{erer} that $\mathcal{V}$ cannot be operator-scalable.
\end{example}

\begin{example}
Let $\{e_{i}\}_{i=1}^{5}$ be the canonical orthonormal basis for $\mathcal {H}_{5}.$ Consider
\begin{align*}
V_{1}&=\textnormal{span}\left\lbrace e_{1}\right\rbrace ,
\\V_{2}&=V_{3}=\textnormal{span}\left\lbrace a_{1}e_{1}+a_{2}e_{2}\right\rbrace ,
\\V_{4}&=\textnormal{span}\left\lbrace a_{3}e_{3}+a_{4}e_{4},e_{5}\right\rbrace ,
\\V_{5}&=\textnormal{span}\left\lbrace a_{5}e_{3}+a_{6}e_{4}\right\rbrace ,
\end{align*}
where $a_{i}\in \Bbb R$ is non zero for all  $1\leq i \leq 5, a_{3}a_{5}+a_{4}a_{6}\neq 0 $ and $a_{3}a_{6}\neq a_{4}a_{5}$. Then $\mathcal{V}=\{V_{i}\}_{i=1}^5$ is a 1-excess fusion frame for $\mathcal {H}_{5},$ which is not weight-scalable as $V_{4}\not\perp V_{5}$. Take
\begin{align*}
U=\begin{pmatrix}
    a_{2}  & -a_{1} & 0  & 0 & 0 \\
    0  & 1  & 0  &  0  & 0  \\
    0  & 0  & a_{4}  & -a_{3} & 0 \\
    0  & 0  & a_{6} &  -a_{5}  & 0 \\
    0  & 0  & 0  & 0 & 1 \\
\end{pmatrix}, 
\end{align*} 
$\gamma_{1}^2=\gamma_{2}^2+\gamma_{3}^2=\left( \frac{1}{a_{2}}\right)^2 $ and $\gamma_{i}=1$ for $i=4,5$. Then Theorem \ref{two equivalent} assures that $\mathcal{V}$ is $(U,\gamma)$-scalable. More precisely, 
\begin{equation*}
\left\lbrace Ue_{1},U(a_{1}e_{1}+a_{2}e_{2}),U(a_{1}e_{1}+a_{2}e_{2})\right\rbrace =\left\lbrace a_{2}e_{1},a_{2}e_{2},a_{2}e_{2}\right\rbrace 
\end{equation*}
is a strictly scalable 1-excess frame for $\mathcal{H}_{1}:=\textnormal{span}\{e_{i}\}_{i=1}^{2}$ with the coefficients $\{\gamma_{i}\}_{i=1}^{3}$ and 
\begin{equation*}
\left\lbrace UV_{i}\right\rbrace _{i=4}^{5}=\left\lbrace \textnormal{span}\{e_{4},e_{5}\},\textnormal{span}\{e_{3}\} \right\rbrace 
\end{equation*}
is an orthogonal fusion Riesz basis for $\mathcal{H}_{1}^{\perp}$.
\end{example}

Finally, we would like to investigate the operator-scalability of all 1-excess fusion frames in $\Bbb R^2.$

\begin{example}
Let $\{f_{i}\}_{i=1}^3$ be a frame of non zero vectors in $\mathcal{H}=\Bbb{R}^2.$ Consider the 1-excess fusion frame $\mathcal{V}=\{V_{i}\}_{i =1}^{3},$ where $V_{i}=\textnormal{span}\{f_{i}\}$. After rotation and reflection of $f_{i}$ around the origin, and re-indexing if necessary, see \cite[Example 3.12]{Remarks}, we can always assume that 
\begin{equation*}
V_{1}=\textnormal{span}\left\lbrace (1,0) \right\rbrace ,~V_{2}=\textnormal{span}\left\lbrace \left( \cos\theta ,\sin \theta \right) \right\rbrace ,~V_{3}=\textnormal{span}\left\lbrace \left( \cos\psi ,\sin \psi \right) \right\rbrace ,
\end{equation*}
with the condition $0\leq \theta \leq \psi \leq\pi.$ Put $U=\begin{pmatrix}
   a  &  0 \\
   b  &  c  \\
\end{pmatrix} $ and $ \gamma=\{\gamma_{i}\}_{i=1}^3,$ where
\begin{equation}\label{tyyrg}
\begin{cases}
a= \dfrac{1}{\sqrt{1+\cos^2 \theta +\cos^2 \psi}},\\
b=\dfrac{-\left( \cos \theta \sin \theta +\cos \psi \sin \psi \right) }{\sqrt{\left( 1+\cos ^2 \theta +\cos ^2 \psi \right) \left( \sin^{2}\theta +\sin^{2}\psi +\sin^2 (\psi -\theta)\right) }}, \\
c= \dfrac{\sqrt{1+\cos^2 \theta +\cos^2 \psi}}{\sqrt{\sin^{2}\theta +\sin^{2}\psi +\sin^2 (\psi -\theta)}},
\end{cases}
\end{equation}
and
\begin{align*}
\begin{cases}
\gamma_{1}=\sqrt{a^{2}+b^{2}},
\\ \gamma_{2}=\sqrt{\left( a\cos\theta \right) ^{2}+\left( b\cos\theta +c\sin\theta \right) ^{2}},
\\ \gamma_{3}=\sqrt{\left( a\cos\psi \right) ^{2}+\left( b\cos\psi +c\sin\psi \right) ^{2}}.
\end{cases}
\end{align*}
Due to $\mathcal{V}$ is complete, then $(\theta,\psi) \not\in \{(0,0), (0,\pi), (\pi,\pi)\}$ and thereby $\sin^{2}\theta+\sin^{2}\psi+\sin^2 (\psi-\theta) > 0.$ It turns out that $\mathcal{V}$ is $(U,\gamma)$-scalable. Indeed,
\begin{align*}
UV_{1}&=\textnormal{span}\left\lbrace (a,b) \right\rbrace ,
\\UV_{2}&=\textnormal{span}\left\lbrace \left( a\cos\theta , b\cos\theta + c\sin\theta \right) \right\rbrace ,
\\UV_{3}&=\textnormal{span}\left\lbrace \left( a\cos\psi , b\cos\psi + c\sin\psi \right) \right\rbrace .
\end{align*}
A detailed analysis reveals that
\begin{align*}
S_{U\mathcal{V}_{\gamma}}&=\sum_{i=1}^{3}\gamma_{i}^{2}\pi_{UV_{i}}
\\&=\begin{pmatrix}
   a^2\left( 1+\cos^{2}\theta + \cos^{2}\psi \right)  &
\begin{aligned}  &ab+a\cos\theta\,(b\cos\theta +c\sin\theta) 
                        \\&+a\cos\psi\,(b\cos\psi +c\sin\psi) \end{aligned}  \\ 
\\ \begin{aligned} &ab+a\cos\theta\,(b\cos\theta +c\sin\theta)
                       \\&+a\cos\psi\,(b\cos\psi +c\sin\psi) \end{aligned} &
\begin{aligned} &b^2+(b\cos\theta +c\sin\theta)^2 
                       \\&+(b\cos\psi +c\sin\psi)^2  \end{aligned} \\
\end{pmatrix}.
\end{align*}
Substituting $a,~b$ and $c$ from \eqref{tyyrg} into the above and employing certain trigonometric relations, we infer that $S_{U\mathcal{V}_{\gamma}}=I_{\mathcal{H}}.$ Therefore $\mathcal{V}$ is $(U,\gamma)$-scalable. Notice that $\mathcal{V}$ is not necessarily weight-scalable for any choice of $\theta$ and $\psi$. For instance, setting $\theta=\frac{\pi}{6}$ and $\psi=\frac{\pi}{2}$, a straightforward computation indicates that $\mathcal{V}$ formed by 
\begin{equation*}
V_{1}=\textnormal{span}\{(1,0)\},~V_{2}=\textnormal{span}\left\lbrace \left( \frac{\sqrt{3}}{2},\frac{1}{2}\right) \right\rbrace ,~V_{3}=\textnormal{span}\{(0,1)\},
\end{equation*}
is not weight-scalable. However, it is $(U,\gamma)$-scalable by $U=\begin{pmatrix}
   \frac{2}{\sqrt{7}}  &  0 \\
   -\frac{\sqrt{42}}{28}  &  \frac{\sqrt{14}}{4}  \\
\end{pmatrix},$
\\$ \gamma_{1}=\frac{\sqrt{10}}{4}, \gamma_{2}=\frac{\sqrt{2}}{2}$ and $ \gamma_{3}=\frac{\sqrt{14}}{4}$.
\end{example}

\textbf{Author contributions.} All authors listed have made equally a substantial, direct, and intellectual contribution to the work and approved it for publication.

\textbf{Conflict of interest.} This work does not have any conflicts of interest.
 
\textbf{Availability of Data and Material.} Not applicable.
 
\bibliographystyle{amsplain}

\begin{thebibliography}{10}

\bibitem{excess of fusion}
E.~Ameli, A.~Arefijamaal and F.~Arabyani Neyshaburi,
\newblock Excess of fusion frames: A comprehensive approach,
\newblock {\em To appear in Asian-European Journal of Mathematics}.
https://doi.org/10.1142/S1793557125500561

\bibitem{scaling weights}
E.~Ameli, A.~Arefijamaal and F.~Arabyani Neyshaburi,
\newblock A survey on constructing Parseval fusion frames via scaling weights,
\newblock {\em arXiv.2502.10244}.



\bibitem{Weaving Hilbert space}
F.~Arabyani Neyshaburi and A.~Arefijamaal,
\newblock Weaving Hilbert space fusion frames,
\newblock {\em Rocky Mountain J. Math.} \textbf{51}(1) (2021), 55-66.


%\bibitem{Aceska17}R.~Aceska, Y.~H,~Kim,\newblock Scalability of frames generated by dynamical operators,\newblock {\em Front. Appl. Math. Stat.}  \textbf{3}(22) (2017), doi: 10.3389/fams.2017.00022.


%\bibitem{2}%@@@@@@@@@@@@@@@@@@@@@@@@@@@@@@
%S.~T.~Ali, J.~P.~Antoine and J.~P.~ Gazeau,
%\newblock Continous Frames in Hilbert Spaces,
%\newblock{\em Ann. Physics,} \textbf{222}(1) (1993), 1--37.


\bibitem{Norm retrieval algorithms}
F.~Arabyani Neyshaburi, A.~Arefijamaal, R. Farshchian and R. A. Kamyabi-Gol,
\newblock Norm retrieval algorithms: A new frame theory approach,
\newblock {\em Math. Meth. Appl. Sci.} \textbf{47}(9) (2024), 7111-7132.



\bibitem{Arabyani dual}%%%%%%%%
A.~Arefijamaal and F.~Arabyani Neyshaburi,
\newblock Some properties of alternate duals and approximate alternate duals of fusion frames,
\newblock {\em Turkish J. Math.} \textbf{41}(5) (2017), 1191-1203.


\bibitem{A.A.SH}  
A.~Arefijamaal, F.~Arabyani Neyshaburi and M.~Shamsabadi,
\newblock On the duality of frames and fusion frames,
\newblock{\em Hacet. J. Math. Stat.} \textbf{47}(1) (2018), 1-10.

%\bibitem{dual Dr}
%A. A. Arefijamaal and E.~Zekaee,
%\newblock Signal processing by alternate dual Gabor frames,
%\newblock{\em Appl. Comput. Harmon. Anal.} \textbf{35} (2013), 535-540.


%\bibitem{Ar01}  A. A. Arefijamaal, F. Arabyani and M. Shamsabadi, Opposite relationships between fusion frames %and their duals, Submitted.


%\bibitem{Bala19}%@@@@@@@@@@@@@@@@@@@@@@@@@@@@@@@
%P.~Balazs, M.~Shamsabadi, A.~Arefijamaal, and A.~Rahimi,
%\newblock U-cross Gram matrices and their invertibility.
%\newblock {\em J. Math. Anal. Appl.} \textbf{476}(2) (2019), 367--390.


%\bibitem{Balazs}
%P.~Balazs, J.~P.~Antoine and A.~Grybos,
%\newblock Weighted and controlled frames: Mutual relationship and first numerical properties,
%\textbf{}\newblock {\em Int. J. Wavelets Multiresolut. Inf. Process.} \textbf{8}(1) (2010), 109–132.
%\bibitem{Ben06} J. Benedetto, A.  Powell  and  O. Yilmaz,   Sigm-Delta quantization and finite frames,  IEEE Trans. Inform. Theory. \textbf{52} (2006), 1990--2005.

%\bibitem{Berut76} F. J. Beutler, W. L. Root, The operator pseudo-inverse in control and systems identifications. In ''Generalized %inverse and applications''. Ed. M. Zuhair Nashed. Academic Press. 1976.


%\bibitem{Bod05}  B. G. Bodmann and,  V. I. Paulsen,  Frames, graphs and erasures,
%Linear. Algebra  Appl. {\bf 404} (2005), 118--146.

%\bibitem{Bol98}  H. Bolcskel, F.  Hlawatsch  and H. G.  Feichtinger,  Frame-theoretic analysis of oversampled filter banks, IEEE Trans. Signal Process. \textbf{46} (1998), 3256--3268.

%\bibitem{Broome}
%H.~Broome, and S.~Waldron,
%\newblock On the construction of highly symmetric tight frames and complex polytopes,
%\newblock{\em Linear Algebra and its Applications,} \textbf{439}(12) (2013), 4135-4151.

 
\bibitem{Cahill-Fickus}
J.~Cahill, M.~Fickus, D.~G.~Mixon, M.~J.~Poteet and N.~K.~Strawn,
\newblock Constructing finite frames of a given sprectrum and set of lengths,
\newblock{\em Appl. Comput. Harmon. Anal.} \textbf{35}(1) (2013), 52-73.
 


%\bibitem{existence}R.~Calderbank, P.~Casazza, A.~ Heinecke, G.~Kutyniok and A.~Pezeshki,\newblock Sparse fusion frames: existence and construction,\newblock{\em Adv. Comput. Math.} \textbf{35}(1) (2011), 1-31.
\bibitem{Remarks}
P.~Casazza, L.~De Carli and T.~T.~Tran,
\newblock Remarks on scalable frames,
\newblock{\em Oper. Matrices} \textbf{17}(2) (2023), 327-342.
%\bibitem{Cas00} P. G. Casazza, The art of frame theory, Taiwanese J. Math. {\bf 4}(2) (2000), 129-202.

\bibitem{CasazzaFickus}
P.~Casazza, M.~Fickus, D.~Mixon, Y.~Wang and Z.~Zhou,
\newblock Constructing tight fusion frames,
\newblock{\em Appl. Comput. Harmon. Anal.} \textbf{30} (2011), 175-187.


\bibitem{Casazza17}%@@@@@@@@@@@@@@@@@@@@@@@@@@@@@@@@@@@@@@@@@@@@@@@@@
P.~Casazza and X.~Chen,
\newblock Frame scalings: A condition number approach,
\newblock{\em Linear Algebra Appl.} \textbf{523} (2017), 152-168.


%\bibitem{Robustness}P.~Casazza and G.~Kutyniok,\newblock Robustness of fusion frames under erasures of subspaces and of local frame vectors,\newblock{\em Contemp. Math.} \textbf{464} (2008), 149-160.



\bibitem{frame of subspace}
P.~Casazza and G.~Kutyniok,
\newblock Frame of subspaces.
\newblock{\em Contemp. Math.} \textbf{345} (2004), 87-114.


%\bibitem{Finite Tight Frames}P.~Casazza and M.~Leon,\newblock Existence and construction of finite tight frames,\newblock {\em J. Concr. Appl. Math.} \textbf{4}(3) (2006).

\bibitem{Furuta}
T.~Furuta, 
\newblock Invitation to Linear Operators: From Matrices to Bounded Linear Operators on a Hilbert Space.
\newblock{\em CRC Press. London}  (2001).


%\bibitem{Chr08}%@@@@@@@@@@@@@@@@@@@@@@@@@@@@@
%O.~Christensen,
%\newblock Frames and Bases: An Introductory Course,
%\newblock {\em Birkh\"{a}user, Boston}, (2008).

%\bibitem{Clark15}%@@@@@@@@@@@@@@@@@@@@@@@@@@@@@@@@@@@@@@@@@@
%C.~A.~Clark, and K.~A.~Okoudjou.,
%\newblock On optimal frame conditioners,
%\newblock {\em Sampling Theory and Applications (SampTA),} International Conference on. IEEE, 2015.

%\bibitem{colgaz10}N.~Cotfas and J.P.~Gazeau.\newblock Finite tight frames and some applications.\newblock {\em J. Phys. A.} \textbf{43}(19) (2010).

%\bibitem{Duffin}%@@@@@@@@@@@@@@@@@@@@@@@@@@@
%R.~Duffin and A.~Schaeffer,
%\newblock A Class of Nonharmonic Fourier Series,
%\newblock {\em Trans. Amer. Math. soc}, \textbf{72} (1952), 341--366.


\bibitem{16}
P.~G\~{a}vru\c{t}a,
\newblock On the duality of fusion frames,
\newblock{\em J. Math. Anal. Appl.} \textbf{333} (2007), 871-879.

\bibitem{dual scalable frames}
B.~Heydarpour, A.~Arefijamaal and F.~Ghaani,
\newblock Characterization of dual scalable frames,
\newblock{\em  Complex Anal. Oper. Theory.} \textbf{18}(3) (2024).


\bibitem{Gitta 13}
G.~Kutyniok, K.~A.~Okoudjou, F.~Philipp and E.~K.~Tuley,
\newblock Scalable frames,
\newblock{\em  Linear Algebra Appl.} \textbf{438} (2013), 2225-2238.


\bibitem{convex geometry}
G.~Kutyniok, K.~A.~Okoudjou and F.~Philipp,
\newblock Scalable frames and convex geometry,
\newblock{\em  Contemp. Math.} \textbf{626} (2014), 19-32.


 
%\bibitem{Leng}J.~Leng, D.~Han, and T.~Huang,\newblock  Probability modelled optimal frames for erasures,\newblock{\em Linear Algebra Appl.} \textbf{438}(11) (2013), 4222-4236. 



\bibitem{Excess 1}
M.~A.~Ruiz and D.~Stojanoff,
\newblock Some properties of frames of subspaces obtained by
operator theory methods,
\newblock{\em  J. Math. Anal. Appl.} \textbf{343} (2008), 366-378. 

\bibitem{Nga}
N.~Q.~Nga,
\newblock Some results on fusion frames and g-frames,
\newblock{\em  Results Math,} \textbf{73}(2) (2018), 1-9. 

\bibitem{Osgooei}
E.~Osgooei and A.~Arefijammal,
\newblock Compare and contrast between duals of fusion and discrete frames,
\newblock{\em Sahand Commun. Math. Anal.} \textbf{08}(1) (2017), 83-96.

\bibitem{Rahimi scalable}
A.~Rahimi and S.~Moayyadzadeh,
\newblock Scalable fusion frames,
\newblock{\em Iran. J. Sci. Technol. Trans. A Sci.} (2025), 1-7.
https://doi.org/10.1007/s40995-024-01765-y

\bibitem{Rahimi}
A.~Rahimi, G.~Zandi and B.~Daraby,
\newblock Redundancy of fusion frames in Hilbert spaces,
\newblock{\em Complex Anal. Oper. Theory.} 
\textbf{10}(3) (2016), 545-565.

\bibitem{mitra sh.}
M.~Shamsabadi, A.~Arefijamaal and P.~Balazs, 
\newblock The invertibility of U-fusion cross Gram matrices of operators.
\newblock {\em Mediterr. J. Math.} \textbf{17}(2) (2020).


%\bibitem{Vale}R.~Vale and S.~Waldron,\newblock Tight frames generated by finite nonabelian groups,\newblock{\em  Numer. Algorithms.} \textbf{48}(1) (2008), 11-27.



%\bibitem{Al00}%@@@@@@@@@@@@@@@@@@@@@@@@@@@@@@@@@@@@@@
 %T.~Bemrose, P. G.~Casazza, K.~Grochenig, M. C. Lammers and R. G ~Lynch,


%\bibitem{12}%@@@@@@@@@@@@@@@@@@@@@@@@@@@@@@@@@
%J.~Benedeto, A.~Powell and O.~Yulmaz,
%\newblock Sigm-Delta Quantization and Finite Frames,
%\newblock {\em IEEE Trans. Inform. Theory,} \textbf{52} (2006), 1990--2005.


%\bibitem{14}
%H.~Boleskel,F.~Hlawatsch and H. G.~Feichtinger ,
%\newblock Frame Theoretic Analysis of Oversampled Filter Banks,
%\newblock{\em  IEEE Trans. Signal Process,} \textbf{46} (1998), 3256--3268.



%\bibitem{16}
%P. G.~Casazza, G.~Kutyniok, S.~Li and C.J.~ Rozell,
%\newblock Modeling Sensor Networks with Fusion Frames,
%\newblock {\em Wavelets XII. San Diego , CA,}
%\newblock{\em  SPIE Proc. SPIE, Bellingham, WA}, (2007), 67011M-1-67011M-11.


%\bibitem{Daubecheis}%@@@@@@@@@@@@@@@@@@@@@@@@@@@@
%I.~Daubecheis, A.~Grossmann and Y.~ Meyer,
%\newblock  Painless Nonorthogonal Expansions,
%\newblock {\em J. Math. Phys}, \textbf{27}(5) (1986), 1271--1283.



%\bibitem{Besselian}%@@@@@@@@@@@@@@@@@@@@@@@
%James R.~Holub,
%\newblock Pre-Frames Operators, Besselian Frames and near-Riesz Bases in Hilbert Spaces,
%\newblock{\em Amer. Math. Soc.}\textbf{122} (1994), 779--785.

%\bibitem{45}
%A. J. E. M.~Janssen,
%\newblock Duality and Biorthogonality for Weyl-Heisenberg Frames,
%\newblock {\em J. Fourier Anal. Appl}, \textbf{1}(4) (1995), 403--436.

%\bibitem{55}
%S.~Mallat,
%\newblock A Wavelet Tour of Signal Processing,
%\newblock {\em Academic Press}, (2009).






\end{thebibliography}

\end{document}